\documentclass[a4paper,11pt]{amsart}

\usepackage{amssymb,amsbsy,amsmath,amsfonts,amssymb,amscd}
\usepackage{latexsym}
\usepackage{mathtools}
\usepackage{graphics}
\usepackage{color}
\usepackage{comment}
\usepackage{soul}
\input xy
\xyoption{all}

\theoremstyle{plain}
\newtheorem{thm}{Theorem}[section]
\newtheorem{theorem}[thm]{Theorem}

\newtheorem{lemma}[thm]{Lemma}
\newtheorem{corollary}[thm]{Corollary}
\newtheorem{proposition}[thm]{Proposition}
\theoremstyle{definition}
\newtheorem{remark}[thm]{Remark}

\newtheorem{defin}[thm]{Definition}

\newtheorem{example}[thm]{Example}

\numberwithin{equation}{section}

\newcommand{\sB}{{\mathcal B}}
\newcommand{\sC}{{\mathcal C}}

\newcommand{\sF}{{\mathcal F}}

\newcommand{\sH}{{\mathcal H}}

\newcommand{\sT}{{\mathcal T}}

\newcommand{\sW}{{\mathcal W}}

\newcommand{\BP}{{\mathbb P}}

\newcommand{\PP}{\ensuremath{\mathbb{P}}}

\newcommand{\CC}{\ensuremath{\mathbb{C}}}
\newcommand{\RR}{\ensuremath{\mathbb{R}}}
\newcommand{\ZZ}{\ensuremath{\mathbb{Z}}}
\newcommand{\QQ}{\ensuremath{\mathbb{Q}}}

\newcommand{\NN}{\ensuremath{\mathbb{N}}}

\newcommand\la{\lambda}

\newcommand\al{\alpha}
\newcommand\be{\beta}
\newcommand\Ga{\Gamma}

\newcommand\ga{\gamma}

\newcommand\e{\epsilon}

\newcommand{\ra}{\ensuremath{\rightarrow}}

\def\eea{\end{eqnarray*}}
\def\bea{\begin{eqnarray*}}

\newcommand\dual{\mathrel{\raise3pt\hbox{$\underline{\mathrm{\thinspace d
\thinspace}}$}}}
\newcommand\qe{\ifhmode\unskip\nobreak\fi\quad $\Box$}       

\def\BOX{\hfill\lower.5\baselineskip\hbox{$\Box$}}

\newtheorem{theo}{Theorem}[section]

\newtheorem{remarkk}[theo]{Remark}

\usepackage{hyperref}
\title[Endomorphisms of the space of elliptic curves]
{Geometric Endomorphisms of the Hesse moduli space of elliptic curves}\author{Fabrizio Catanese, Edoardo Sernesi}
\date{August  24 2023}
\address {Mathematisches Institut der Universit\"at Bayreuth\\
NW II,  Universit\"atsstr. 30\\
95447 Bayreuth}
\email{fabrizio.catanese@uni-bayreuth.de}
\address{  Korea Institute for Advanced Study, Hoegiro 87, Seoul, 
133--722.}
\address{  Universit\'a di Roma Tre, 
L.go S.L. Murialdo 1, 00146 Roma (Italy).}
\email{sernesi@gmail.com}
\thanks{
 The authors   have no conflicts of interest to declare that are relevant to the content of this article.}

\begin{document}

\maketitle

\begin{abstract}
We consider the  geometric map $ \mathfrak C$, called Cayleyan, associating to a plane cubic
$E$  the adjoint of its dual curve.   We  show that  $ \mathfrak C$ and  the classical Hessian
map $ \mathfrak H$ generate a free semigroup.

We begin the investigation of the geometry and dynamics of
these maps, and of the {\bf geometrically special elliptic curves}: these are the 
 elliptic curves isomorphic to cubics in the Hesse pencil which are fixed by some endomorphism belonging to the
semigroup  $\sW(\frak H,  \frak C)$ generated by $ \frak H,  \frak C$. 

 We point out then how the dynamic behaviours of $ \mathfrak H$ and $  \mathfrak C$
differ drastically.   Firstly, concerning the number of real periodic points:
for $ \mathfrak H$ these are infinitely many, for $  \mathfrak C$ they are just $4$.
Secondly, the Julia set of  $ \mathfrak H$ is the whole projective line, unlike what 
happens for all
elements of $\sW(\frak H,  \frak C)$ which are not iterates of $ \mathfrak H$.
\end{abstract}

\tableofcontents

\section{Introduction: the Hessian and the Cayley  construction
and the main results}\label{S:introd}

Denote by  $\BP^2$ and by $\BP^{2\vee}$ the projective plane and its dual, and consider a nonsingular plane cubic 
$E\subset \BP^2$,  with equation $F(x)=0$. 

 We shall indeed identify $\BP^2$ and $\BP^{2\vee}$ in terms of the standard scalar product.

We can then associate to $E$ another cubic  in  the following two ways. 

The first way is classical  and well known: we let $ \mathfrak H (E)$ be the Hessian cubic, that is, the  zero set of the determinant $ det (\frac{ \partial^2 F}{\partial x_i \partial x_j})$
of the Hessian matrix of $E$.

In the second way we let 
  $C\subset \BP^{2\vee}$ be the sextic dual to $E$. Then we let $ \mathfrak C (E) \subset \BP^{2\vee}$ be the cubic adjoint to $C$, namely the unique cubic containing the 9   cusps of $C$. 
  
  We can now repeat the construction starting from $D : = \mathfrak C (E)$. 
  We   consider the sextic $D^\vee\subset \BP^2$ dual to $D$ and its adjoint cubic $E_1 : = \mathfrak C  \circ \mathfrak C (E) \subset \BP^2$. 

   These constructions can be iterated, and our general theme is the study of the iteration of the mapping $\mathfrak C $  and of other holomorphic maps which are compositions of $\mathfrak C $ and $ \mathfrak H$.
  
In particular we associate to $E$ a sequence of  cubics 
(which are not necessarily nonsingular, as the case of the Fermat cubic shows):
  $$
  E_1= \mathfrak C  \circ \mathfrak C (E), E_2 := \mathfrak C  \circ \mathfrak C (E_1), E_3, \dots \subset \BP^2.
  $$

 These were the motivating questions which started our investigations:

\medskip

\textbf{Question (1):}\label{Q1} Find explicit  conditions to  have   $E=E_n$  for a given $n \ge 1$ (the case $n=1$,  due to \cite{hollcroft}, shall be explained in the sequel).

\textbf{Question (2):} What is the geometry of the iteration of the holomorphic maps $\mathfrak C, \mathfrak H$  and of the  holomorphic maps in the semigroup $  \sW (\frak H, \frak C)$  generated by $\mathfrak C, \mathfrak H$?

\medskip

 The  best way to answer these questions consists in  reducing the number of parameters for the cubic curve $E$, using the fact that every smooth plane cubic is projectively equivalent to 
 a cubic in the Hesse pencil of cubic curves, parametrized by $\la \in \PP^1$, and which is a moduli space 
 for elliptic curves with level $3$ structure.
 
  Then we can view $\mathfrak C, \mathfrak H$ as rational self-maps of degree $3$
 of $\PP^1$, and try to study their dynamics.
 
 \smallskip
 
 These are our main  results:
 
  \smallskip
  
 {\bf  Theorem \ref{free-semigroup}:} {\em $ \mathfrak H $ and $ \mathfrak C$
 generate a free rank 2 semigroup   $ \sW (\frak H, \frak C)$ of self-maps of $\PP^1(\CC)$.
 }
 
  \smallskip
  
 {\bf  Theorem \ref{D-periodic}:}
{\em 
The real periodic points of $\mathfrak C$ are just the four fixpoints of $\mathfrak C$:
$$
  \infty, -\frac{1}{2}, -\frac{1}{2}\pm \frac{\sqrt{3}}{2}
  $$

 }

 \smallskip
  
 {\bf  Theorem \ref{H-periodic}:}
{\em 
 $\mathfrak H$ has infinitely many real periodic points.
 Indeed, $\mathfrak H^{\circ 2m}$  has, for $ m \geq 2$,  at least $ 2(3^m -1)$ real fixpoints.
 
Likewise, the lines $\e \RR, \e^2 \RR \subset \CC \subset \PP^1(\CC)$, where $\e$ is a primitive third root of  $1$, contain each infinitely many  periodic points.

 }
 \smallskip
 
  {\bf  Theorem \ref{julia}:}
{\em
The Julia set $J(\frak H)$ 
equals the whole projective line $\PP^1(\CC)$.

Also, for each endomorphism  $ f \in \sW (\frak H, \frak C)$, 
the Julia set $J(f)$ equals the whole projective line $\PP^1(\CC)$
if and only if $f$ is  an iterate of $\frak H$.
}

\section{The Hesse pencil, the Hesse group and equations for the Hessian and Cayleyan maps}

 \smallskip
 
 The Hesse cubic $E (\la)$ is the cubic 

 $$
E (\la):\ \la_0 \sum_i x_i^3 + 6 \lambda_1  x_1x_2x_3 =0, \quad \la  : =
\frac{\la_1}{\la_0} \in \PP^1(\CC) = \CC \cup \{ \infty\}.
$$

  The Hesse pencil was introduced by Hesse in 1844 (formula 54 of page 90 of \cite{hesse}) and still plays a role for many 
 recent developments, see for instance  \cite{b-l-13},  \cite{Art-Dolg},  \cite{ca-se}.
 
  The advantage of using the Hesse pencil is twofold: every nonsingular plane cubic is projectively equivalent to a cubic in the Hesse pencil; moreover, and more important, is that for the  cubics of the Hesse pencil   we can explicitly write
 the endomorphisms $\mathfrak H$ and $\mathfrak C$.
 
  We shall briefly recall some classical results, which can be found for instance
 in \cite{e-c} Libro 3, pages 213-240  (here  is explained in detail the Cayley construction, for each plane curve, see also the historical note on pages 185-188; the
 Cayleyan  is also briefly touched in \cite{coolidge}, 
 page 150 and can be found in \cite{cayley-omnia}), and in \cite{b-k} pages 283-305 (of the English edition),
 and  \cite{Art-Dolg}.

 The only caveat is: each author has a different parameter $\la' : = 6 \la = :  - 3 \mu$,
 hence one has to be careful in using the various formulae.

For each nonsingular cubic $E$ the cubic and the Hessian cubic 
$\mathfrak H(E)$ intersect in the nine inflection points of $E$. 

If $E$ is in the Hesse pencil, these nine points  form  the base point set $\sB  $ of the 
pencil, 
$$  \sB  : = \{ x \in \PP^2 | \exists j \in \{1,2,3\} s.t. \ x_j = \sum_i x_i^3=0\}.$$
Hence either a curve $E (\la)$ in the Hesse pencil has exactly 9 flexpoints
( the base points in  $ \sB = E (\la) \cap \mathfrak H (E (\la) )$), 
or $E (\la) = \mathfrak H (E (\la) ) $, and then all points are flexpoints, which means that
$E (\la) $ is a union of 3 lines.

  It is well known  and easy to compute that $\mathfrak H (E (\la) ) = E ( \mathfrak H (\la)) $, where,  writing $\lambda =\frac{\lambda_1}{\lambda_0}$ if $\lambda_0\ne 0$,
$$ \mathfrak H_0 = - 6 \la^2, \mathfrak H_1 = 1 + 2 \la^3 \Rightarrow 
 \mathfrak H (\la)  = - \frac{1 + 2 \la^3}{  6 \la^2}.$$
 
All curves in the Hesse pencil are smooth, with exception of 
 the curves with 
 $$ \la = \mathfrak H (\la)  \Leftrightarrow \la = \infty, \ {\rm or } \ (2\la)^3 = -1.$$
 
 These four cubics are the 4 triangles in the Hesse pencil.

  Direct inspection of the equations of the cubics in the  Hesse pencil
shows invariance of each equation by the symmetric group $\mathfrak S_3$, and by the group  $U_3$ of
 unimodular diagonal matrices with entries cubic roots of $1$.
 
 These two subgroups generate the so called extended Heisenberg group
 $\widehat{\sH}$, of order $54$ (see for instance \cite{b-l-13}).
 While $U_3$ and the cyclical permutation  of coordinates generate the Heisenberg group
 $\sH$, of order  $27$.
 
  The extended Heisenberg group $\widehat\sH$ acts on each individual cubic through
   $\widehat G$,  the group of order $18$ which is  the quotient of  $\widehat\sH$ by its centre
 (consisting of multiples of the identity matrix by a cubic root of $1$),  acting   on each individual 
 cubic of the pencil as the 
group generated by translations by 3-torsion elements, and by multiplication by $(-1)$.

Instead the quotient of the Heisenberg group $\sH$ by its centre 
is isomorphic to $(\ZZ/3)^2$, the group of translations
by 3-torsion elements.

One sees also an order 3 transformation of $\PP^2$
$$ \ga :  (x_1,x_2, x_3) \mapsto (x_1,\e x_2, \e^2 x_3), $$
where $\e$ is a primitive third root of unity;
$\ga$ has the effect of sending $E(\la)$ to $E(\e \la)$, and together the two transformations
$$ \ga (\la) = \e \la, \ \phi (\la) = \frac{1-\la}{2 \la+1}$$
generate a subgroup of the projectivities, which we shall denote
$\mathfrak A_4$,  and call the {\bf  Alternating Hesse group}, since it permutes the four triangles in the Hesse pencil
as the alternating group in 4 letters ( $\phi$ is the double transposition permuting
$\infty$ with $-\frac{1}{2}$, and $- \frac{1}{2} \e$ with $- \frac{1}{2} \e^2$).

Also $\phi$ is induced by a projectivity which leaves the Hesse pencil invariant
(see \cite{b-k}, page 297),  given by the Lagrange resolvents:
$$ \Phi : (x_1, x_2, x_3) \mapsto (x_1+ x_2+ x_3, x_1+ \e x_2+ \e^2 x_3,
x_1+ \e^2 x_2+ \e x_3).$$

Indeed $\mathfrak A_4$ is the quotient of the  larger group, called the 
extended Hesse group $\widehat{Hes}$, of order $648$, generated 
by  $\widehat{\sH}, \Phi, \ga$. 

$\widehat{Hes}$ divided by the centre yields the Hesse group $Hes$ of projectivities
of $\PP^2$, and we have:
$$ Hes \cong SA(2, \ZZ/3), \ \ \mathfrak A_4 = \widehat{Hes} / \widehat{\sH} \cong SL(2, \ZZ/3),$$
where $SA$ denotes the special affine group.

\bigskip

 We shall now calculate and see that $\mathfrak C (E (\la) ) = E ( \mathfrak C (\la)) $, where

 $$
\frak C(\lambda) = \frac{1-4\lambda^3}{6\lambda}. 
 $$

  \subsection{Equations for $\mathfrak C$}

Choosing as above  homogeneous coordinates  $(x_1,x_2,x_3)$  in $\BP^2$ so that $E$ can be written in Hesse normal form:
$$
E(\la):\ \sum_i x_i^3 + 6\lambda x_1x_2x_3 =0, \quad \lambda\in \CC,
$$
we denote by $(y_1,y_2,y_3)$ the dual coordinates in $\BP^{2\vee}$.  Since the operations of duality and adjunction are equivariant with respect to the action of $\widehat G$, the quotient of the extended Heisenberg group $\widehat\sH$ by its centre, the cubic $D$ can be also written in Hesse normal form in $y_1,y_2,y_3$:

$$
D= E(\frak C)= \ \sum_i y_i^3+6\frak C y_1y_2y_3 = 0,  \quad \text{for some $\frak C=\frak C(\lambda)\in \PP^1$}
$$

In order to compute  $\frak C(\lambda)$ we proceed as follows.

The curve $D$ is determined by the condition of passing through the 9 singular points of $C$, dual of the   9 inflectional tangents of $E$. Since these 9 points are an orbit of the group $\widehat G$, and since    $\widehat G$ acts on $D$, it suffices to impose the condition that $D$ contains just one of the 9 points. We choose the flexpoint $(1,-1,0)$ of $E$. The partial derivatives of the Hesse polynomial of $E$ are:
$$
(x_1^2+2\lambda x_2x_3,x_2^2+2\lambda x_1x_3, x_3^2+2\lambda x_1x_2)
$$
and therefore the coordinates of the inflectional tangent at 
 $(1,-1,0)$ are  $(1,1,-2\lambda)$. Imposing to $D$ to contain this point we obtain the condition: 
 
$$
2-8\lambda^3-12\lambda\frak C=0
$$
which gives:

\begin{equation}\label{E:mulambda}
 \frak C (\lambda)    = \frac{1-4\lambda^3}{6\lambda}  
\end{equation}

Thus we get the following expression for $D$:
\begin{equation}\label{E:D}
  D:\ \lambda\left(\sum_iy_i^3\right) + (1-4\lambda^3)y_1y_2y_3=0.  
\end{equation}

 \subsection{Symmetry by the Hesse Alternating group $ \mathfrak A_4$}
 
 The previous  discussions explain the  underlying reasons for the following

\begin{lemma}
Let $g \in \mathfrak A_4$: then we have

(1) $\mathfrak H \circ g = g \circ \mathfrak H$.

(2)  $\mathfrak C \circ g = \bar{g} \circ \mathfrak C$,
where $\bar{g}$ is the projectivity where one takes the complex conjugate
coefficients of those of $g$.
\end{lemma} 
\begin{proof}
Since $\mathfrak A_4$ is generated by $\ga$ and $\phi$, and $\phi$ is real, 
it suffices to verify that $\mathfrak C, \mathfrak H$ commute with $\phi$,
and, concerning $\ga$,  one sees right away that 
$$\mathfrak H (\e \la) = \e \mathfrak H ( \la), \ \ \   \mathfrak C (\e \la) = \e^2 \mathfrak C ( \la).$$ 
\end{proof}

\begin{corollary}
The set of fixed points for $\mathfrak H^{\circ n}$ consists of $\mathfrak A_4$-orbits.

The set of fixed points for $\mathfrak C^{\circ n}$ consists of $\mathfrak A_4$-orbits
if $n$ is even.
\end{corollary}
\begin{proof}
$\mathfrak H \circ g = g \circ \mathfrak H$ implies 
$\mathfrak H^{\circ n} \circ g = g \circ \mathfrak H^{\circ n}$.

It is obvious that if $f \circ g = g \circ f$, then $f(z) = z \Rightarrow 
f(g z) = gz$. 

Similarly for the second assertion.

\end{proof}

The previous Corollary justifies the interest for the special  $\mathfrak A_4$-orbits,
those with less than $12$ elements (and hence corresponding to curves with extra automorphisms).

\begin{remark}
The special $\mathfrak A_4$-orbits in the Hesse pencil are:
\begin{itemize}
\item
the orbit $\sT$ of $\infty$, consisting of the four triangles;
\item
the orbit $\sF$ of $\la=0$, which consists of $\{ 0, 1 , \e, \e^2\}$,
that is, the set of {\bf equianharmonic} (Fermat) elliptic curves;
\item
the set $Harm$ of {\bf harmonic} elliptic curves: this set has 6 elements, 
 is the orbit
of the two fixed points for $\phi$, and consists of  the six points:
$$ \frac{1}{2}  \e^i  ( - 1 \pm \sqrt{3}) , \ i=1,2,3.$$
\end{itemize}

The respective equations are:
$$\sF  = \{ \la |  
\la (\la^3 -1)=0\}$$ for the four Fermat curves,
and since 
$$(2 \la ^2 + 2 \la -1 ) ( 2 \e^2 \la ^2 + 2 \e \la -1 )  ( 2 \e  \la ^2 + 2  \e^2 \la -1 ) = 
 8 \la ^ 6 + 20 \la^3  -1$$

$$ Harm = \{ \la| 
 8 \la ^ 6 + 20 \la^3  -1 =0\}$$
 for the six harmonic curves.
\end{remark}

A direct consequence of the knowledge of the special orbits is the following
result of Hollcroft \cite{hollcroft}:

\begin{corollary}\label{harmonic}
The fixed points of $\mathfrak H$ are the points of $\sT$, and the set of periodic points  of period exactly 2 for $\mathfrak H$ (these are the fixpoints of $\mathfrak H^{\circ 2}$ which are not fixpoints of $\mathfrak H$) is the set $Harm$ of the 6 harmonic cubics. 
Instead, the fixed points of $\mathfrak C$ are the real points of  $\sT\cup Harm$, while the set of periodic points of period  exactly 2 for $\mathfrak C$ is the set of non-real points of $\sT\cup Harm$. 
\end{corollary}
\begin{proof}
We borrow from Section 3 that a rational map of degree $d$ has precisely 
$(d+1)$ fixpoints, counted with multiplicity.

The  periodic points of period $2$ for $\mathfrak H$
are  the fixpoints of $\mathfrak H^{\circ 2}$ which are not fixpoints of
$\mathfrak H$. Hence we get $3^2 + 1 - (3+1)=6$ points.

We have a union of $\frak A_4$-orbits, where obviously the points of the orbit have the same
multiplicity. Since there are only orbits with $4$, $6$ or $12$ elements,
the only possibility is an orbit of $6$ elements, counted with
multiplicity $1$. And then we know that  this orbit is the orbit of harmonic cubics.

A similar argument applies  for $\mathfrak C$, showing that the fixed points for  $\mathfrak C^{\circ 2}$ are again the points of $\sT\cup Harm$. One then observes that the fixed points of $\mathfrak C$ are $\infty$ and the 3 real roots of  $4\lambda^3+6\lambda^2-1=0$  (see \eqref{E:fixptsC}).
\end{proof}

\begin{remark}
For completeness we also give the equations for the $j$-invariant
 of the curves in the Hesse pencil (see \cite{b-k}, Theorem 10, page 302):

$$ j(\la) = \frac{ 8 \la^3 ( 8  - 8 \la^3)^3  } { ( 8 \la^3  + 1)^3 }.$$

Here, $ j = 0$ for the equianharmonic curves, $j = \infty$
for the four triangles, $ j = 4^3=64$ for the harmonic cubics.

 Enriques and Chisini use instead another nonstandard form of the $j$-invariant,
for which $j = \infty$ for the harmonic cubics.

\end{remark}

We finally observe the following:

\begin{remark}
 The transformations $\mathfrak H, \mathfrak C$
are related by the involutory projectivity
$$ \iota (\la) : = - \frac{1}{2 \la},$$ namely 
$$  \mathfrak C (\la) =  \mathfrak H ( - \frac{1}{2 \la}) = \mathfrak H \circ \iota (\la).$$

In particular the semigroup generated by $\mathfrak H, \mathfrak C$ is contained in the semigroup generated by $\mathfrak H, \iota$.

\end{remark}

\begin{proposition}
The Alternating Hesse group $\frak A_4$ and the involution $\iota$ generate
a subgroup of order $24$ of $\PP GL(2, \CC)$, isomorphic to $\mathfrak S_4$,
called the {\bf Symmetric Hesse group} $\mathfrak S_4$.
\end{proposition}
 \begin{proof}
 We know that $\frak A_4$ is generated by $\phi, \ga$.
 
 Direct calculation shows that
 $$ \iota \phi = \phi \iota , \ \iota \ga \iota = \ga^2.$$
 
 Hence $\iota$ normalizes $\frak A_4$ and we get a semidirect product of 
 $\frak A_4$ with $\ZZ/2$. This is isomorphic to $\frak S_4$,
 as follows: just send 
 $$ \phi \mapsto (1,2)(3,4) , \ \ga \mapsto (1,2,3), \ \iota \mapsto (1,2).$$
 \end{proof}

 \begin{remark}
 To understand the action of $\frak S_4$, since we know the action of $\frak A_4$,
 it suffices to notice that $\iota$ transposes the pairs
 $$ \{0, \infty\}, \{ 1, -\frac{1}{2}\}, \{ \e,  - \frac{\e^2}{2}\}, \{ \e^2, - \frac{\e}{2}\},$$
 in particular exchanging $\sT$ (triangles) with $\sF$ (Fermat).
 
 While $\phi$ leaves $\sT, \sF$ respectively invariant and transposes the pairs
 $$ \{0,  1\}, \{ \infty, -\frac{1}{2}\}, \{ \e,  \e^2 \}, \{ - \frac{\e^2}{2}, - \frac{\e}{2}\}.$$
 
 We get in this way an embedding of $\frak S_4$ into $\frak S_8$.
 
 $\frak S_4$ preserves the set $Harm$ of Harmonic points, and we get in this
 way an embedding of $\frak S_4$ into $\frak S_6$\footnote{This corresponds to the conjugacy action of $\frak S_4$ on the $4$-cycles.}.
 
 For instance to $\iota$ corresponds a product of 3 transpositions, since
 $$ \frac{1}{2} \e^i ( -1 \pm \sqrt{3}) \mapsto  \frac{1}{2} \e^{-i} ( -1 \mp \sqrt{3}).$$
 
 \end{remark}

  \begin{remark}
 We consider now  the problem of the description of the action of $\frak H$ on
 the j-invariant, that is the expression of the action of  $\frak H$ on
 $\PP^1 / \frak A_4$.
 
 To this purpose, set $ \mu : = 8 \la^3$, so that 
 $$ j(\la) = \frac{ \mu (8 - \mu)^3}{( \mu +1)^3 }.$$
 
 Then, setting $h : = \frak H (\la)$, then $ \mu' : =  8 h^3 = - \frac{( 4 + \mu)^3}{3^3 \mu^2}$,
 hence
 $$ \frak H (j) = \frac{ \mu' (8 - \mu')^3}{( \mu' +1)^3} =  \frac{ (4+\mu)^3 ( (4+\mu)^3+ 6^3 \mu^2)^3} { 3^3 \mu^2 ((4+\mu)^3-3^3 \mu^2 )^3} .$$
 
 Observing that under  $ \frak H$ the inverse image of $\infty$ is $0$ counted with multiplicity $2$ (corresponding to the Fermat cubics, for which $j=0$, and $\infty$ with multiplicity $1$, corresponding to the triangles) we see that there are constants $d_i$ such that
 $$ \frak H (j) =  \frac{ d_3 j^3 + d_2 j^2 + d_1 j + d_0}{3^3 j^2}. $$
 
 Substituting the expression of $j$ as a rational function of $\mu$, we can determine the  constants  $d_i$ and obtain
 (in performing the calculation, after clearing denominators, it is convenient to set first $\mu=0$, in order to obtain $d_0$, and  then $\mu = -1$, 
 in order to obtain $d_3$: then one has to look at the coefficients of $\mu^{10}$ and then of $ \mu^9$):
 $$  d_0 = 2^{24} , \ d_3 = -1, \ d_2 = 3 \cdot 2^8 , \ d_1 =  -164352 =  - 3 \cdot 2^{16} , \  $$
 hence
 $$ \frak H (j) =   - \frac{   (j - 2^8)^3 } {3^3 j^2}. $$
 A similar calculation was performed in \cite{poppam}, Thm. 3.4.
 \end{remark}

 \begin{remark}
 We believe that the action of $\frak H$ does not descend to 
 $\PP^1 / \frak S_4$.
 
 The reason is, $\frak H$ commutes with every element of $\frak A_4$.
 And every automorphism of $\frak S_4$ which is the identity on $\frak A_4$
 must be the identity: because $\phi = (1,2)$ would be  mapped to another 
 transposition $\tau$ commuting with $(1,2)(3,4)$ and conjugating $\ga = (1,2,3)$
 to its inverse. The first property implies that if $\tau$  is not $(1,2)$, it is $(3,4)$;
 but the latter does not conjugate $\ga$ as required.
 
 Then we observe that $\frak H$ does not commute with $\iota$, since
 $$  \iota \circ \frak H (\la) = \frac{3 \la^2} { 1 + 2 \la^3} , \ \ 
 \frak H \circ  \iota (\la) = 
 \frak C (\la) = \frac{1 - 4 \la^3} { 6 \la}.$$
\end{remark}

\subsection{The original problem}

We now  consider 
$$ \frak C^{\circ 2} (E(\la)) = E (\frak C^{\circ 2} (\la)) $$ 
and  by the previous computation
we have:

$$
\frak C^{\circ 2} (\la)= \frac{1-4\frak C(\lambda)^3}{6\frak C(\lambda)}
= \frac{\lambda}{1-4\lambda^3} - \frac{(1-4\lambda^3)^2}{54\lambda^2}.
$$

We   say that   $\lambda\in \CC $ is  \emph{$\frak C^{\circ 2}$-periodic}
if for some $n \geq 1$
 we have $\la = \frak C^{\circ 2n} (\la)$.
  The smallest $n$ for which this occurs is called the  $\frak C^{\circ 2}$-\emph{period} of $\lambda$. 
 
 The condition $E(\la) =  \frak C^{\circ 2} (E(\la)) $  translates into the equation
  $\lambda = \frak C^2 (\lambda)$.

It is satisfied by the fixpoints of $\frak C$, 
$$
 \infty, -1/2,\quad  -1/2 \pm\sqrt{3}/2
 $$
and, in view of Corollary \ref{harmonic},
by the six harmonic cubics.

We  shall soon show  that the number of  periodic points is always infinite for each holomorphic function  $f$  in our semigroup. 

We leave aside for the time being the problem of the exact determination of the cardinality of the set 
$\Lambda_n$, defined as the set of periodic points of $\frak C^2$
of period exactly $n$.


 \subsection{The geometric semigroup}

 \begin{defin}
 
 (I) Let $\sW(h,c)$ be  the semigroup freely generated by the letters $h,c$,
and let  $W(h,c) $ be a word in   $\sW(h,c)$.
 
(II)  Consider the homomorphism $\psi : \sW(h,c) \ra End(\PP^1)$
 defined by
 
 $$ \psi (h) : = \frak H, \ \ \ \psi (c) : = \mathfrak C,$$
 and denote its image by 
 $$ \sW(\frak H,\frak C) : = Im (\psi).$$
 
 \end{defin}
   We observe that both endomorphisms $\mathfrak H$ and $\mathfrak C$
 are degree $3$ maps of $\PP^1$ to $\PP^1$, hence   the following is clear:
 
 \begin{remark}
 
 Each holomorphic function $f$ in the semigroup, $f : = \psi (W(h,c) )$ has degree $3^e$, where $e$ is the exponent (that is, the length) of the word $W(h,c) $.  

In particular, $f$ has $3^e + 1$ fixpoints counted with multiplicity, 
 in view of the forthcoming  proposition in the next section.
 \end{remark}
 
 \section{Generalities on periodic points and real periodic points of rational maps}
 
 \begin{proposition}\label{P:fixpts}
 Let $f: \PP^1 \ra \PP^1$ be a holomorphic map of degree $m$. Then:
 
 (i) The number of fixpoints of $f$, counted with multiplicity, equals $ m+1$.
 
 (ii) The multiplicity of a fixpoint $P$ is calculated as follows, in terms of   a local coordinate $z$ for which $z(P)=0$: consider the Taylor development of $f$ at $P$,
 $$
 f(z)= a z + b z^h + {\rm higher \ order \ terms},$$
 where $b\ne 0,  \  h>1$.

 Then the multiplicity as a fixed point equals the multiplicity of $ f(z) -z$ at $P$, which is $1$ for  $a\ne 1$,
 otherwise it  equals $h$.
 
 (iii) If $ m >1$ the cardinality of the set $Per (f)$ of  periodic points of $f$, 
 $Per (f): = \{ P | \exists r \geq 1, f^{\circ r} (P) = P\}$, is infinite.
 
 \end{proposition}
 \begin{proof}
 (i) follows since in $\PP^1 \times \PP^1$ the intersection number of the diagonal (a divisor of bidegree $(1,1)$) with the graph of a morphism of degree $m$ (a divisor of bidegree $(1,m)$) equals $m+1$.
 
 (ii) at the point $(P,P)$ we must calculate the intersection multiplicity of two curves
 which, in local coordinates, are given by the respective equations:
 $$ z_2 = z_1 , z_2 = f(z_1) ,$$
 equivalently,
$$  z_2 = z_1 , f(z_1)- z_1 =0.$$
 Hence the intersection multiplicity equals the multiplicity of the function 
$$ f(z) - z = (a-1) z + b z^h + {\rm higher \ order \ terms}.$$

 The assertion follows then right away.
 
 (iii)  Assume by contradiction that the set $Per(f)$ is finite, and consists of $k$
 points $P_1, \dots, P_k$.
 
 Let $n_i$ be the period of $P_i$, namely the smallest exponent ($\geq 1$) such that 
 $ f^{\circ {n_i}} (P_i) = P_i$.
 
 Take then $n$ to be a common multiple of the $n_i$'s,
 and denote $F : = f^{\circ {n}}$.
 
 Now, at the point $P_i$, which is fixed for each iterate $F^{\circ r}$ of $F$,
 we claim that there are three possibilities, writing as above at $P_i$ 

$$F(z) = a' z + b' z^{h'} + {\rm higher \ order \ terms}:$$

 (I) $a'=1$, and $P_i$ counts with the same multiplicity $h'$ for all the iterates
 of $F$;
 
 (II)  $a' \neq 1$ and $a'$ is not a root of unity:
 then $P_i$ counts with the same multiplicity $1$ for all the iterates
 of $F$;
 
 (III)  $a'$ is a primitive root of unity of order $c_i >1$, and then $P_i$ counts with the same multiplicity $1$ for all the iterates
$F^{\circ r}$ such that   $c_i$ does not divide $r$.

While (II), (III) are obvious, (I) needs the observation that then
 $$ F^{\circ r} (z) = z + r b' z^{h'} + {\rm higher \ order \ terms}.$$

{\bf Conclusion:} take an integer $r > 1 $ which is not divisible by any of the $c_i$'s of case (III),
and consider  $ F^{\circ r}$. By assumption, its only  fixpoints 
are the points $P_1, \dots, P_k$.  For each one of them the multiplicity as a fixpoint of
$F$ equals the one as a fixpoint of $ F^{\circ r}$.

This however contradicts statement (i), since the sum of the multiplicities for
$F$ equals $ deg(F) + 1= m^n + 1$, while for $ F^{\circ r}$
the sum of the multiplicities equals 
$$ deg(F^{\circ r}) + 1 = m^{nr} + 1 > m^n + 1.$$

 \end{proof}

 We shall now consider the case where $f$ (as it is the case for $\mathfrak H, \mathfrak C$)
 is a real rational function $f : \PP^1 \ra \PP^1$ of degree $m \geq 2$.
 
 Then $f (\PP^1(\RR) ) = \PP^1(\RR)  $, and we give the following definition.
 
 \begin{defin}
 If $f \in \RR (\la)$, then its real degree $r$ is defined through the equality 
 $$ H_1(f) [\PP^1(\RR)] = r [\PP^1(\RR)],$$
 where the fundamental class $[\PP^1(\RR)]$ is the positive generator of
 $H_1 (\PP^1(\RR), \ZZ)$.
 \end{defin}
 
 By definition $r$ is the intersection number of the 
  subset $ \PP^1(\RR) \times \{y\}
 \subset \PP^1(\RR) \times \PP^1(\RR)$ with the 
 graph of the restriction
 $f_r$ of $f$ to $\PP^1(\RR)$.
 
 It counts the solutions $x$ of the equation $ f(x) = y, \ x \in \PP^1(\RR)$ with sign and multiplicity.
 We shall later see that for $\mathfrak H, \mathfrak C$  one has $ r=-1$.
 
 We have an easy lemma
 
 \begin{lemma}\label{sign-fixpoints}
 (i) The number of real fixpoints (that is, the fixpoints in $\PP^1(\RR)$) of a real rational
 function $f$, counted with sign and multiplicity, equals $ r-1$, where $r$ is the real degree of $f$.
 
  (ii) The multiplicity of a fixpoint $P$ is calculated as follows, in terms of   a local coordinate $z$ for which $z(P)=0$: consider the Taylor development of $f$ at $P$,
$$ f(z) = a z + b z^h + {\rm higher \ order \ terms},$$
 where $ b \neq 0$ and $h >1$. 
 
 Then the multiplicity (with sign) as a fixed point equals the sign of  $(a-1) \in \RR$ for $a \neq 1$,
 
 otherwise it  equals zero for $h$  even, and for $h$ odd it is the sign of  $b \in \RR$.

 \end{lemma}
 \begin{proof}
 (i) The number of fixpoints equals the intersection number of  the diagonal of 
 $ \PP^1(\RR) \times \PP^1(\RR)$ with the graph of the restriction
 $f_r$ of $f$ to $\PP^1(\RR)$. The two classes are of respective type $(1,1) $,
 $(1,r) $, 
 hence their intersection number equals $ r-1$.
 
 (ii) is proven as in Proposition \ref{P:fixpts}: we reduce to calculate the 
  signed multiplicity of  $ f(z) -z = (a -1) z + b z^h  + h.o.t $ at $P$.
  
  For  $a>1$ the function $f(z)-z$ is increasing, for  $a<1$ the function $f(z)-z$ is decreasing. For
   $a=1$, and $h$ even, we obtain  $bz^h +h.o.t. $, which is homotopic to  $ b z^h -
  \e \frac{b}{|b|}$
  and has no zeros. For $h$ odd, the function is increasing or decreasing
  according to the sign of  $b$.
 
 \end{proof}
 
 Clearly also the real degree is multiplicative for compositions of real rational maps.
 
 Hence we obtain, with almost identical proof,  a similar result to Proposition \ref{P:fixpts}:
 
 \begin{proposition}
  Let $f: \PP^1 \ra \PP^1$ be a real rational  map of real degree $r$. 
  
  If $ |r|  >1$ the cardinality of the set $Per (f)(\RR)$ of  real periodic points of $f$, 
 $$Per (f)(\RR): = \{ P \in \PP^1(\RR) | \exists n \geq 1, f^{\circ n} (P) = P\},$$
  is infinite.
 
 \end{proposition}

 \begin{remark}
 We shall later see that, for the map $\mathfrak C$, there are 4 fixpoints, 
 for two of them  $a=0$, for the other two  $a=\sqrt{3}, - \sqrt{3}$.
 Hence the number of real fixpoints counted with sign as explained in  Lemma \ref{sign-fixpoints} equals $-3+1 = -2$
confirming that $r=-1$.

The interesting fact is that we shall prove in Theorem \ref{D-periodic} that these 4 fixpoints are
all the real periodic points for $\mathfrak C$.
 
 \end{remark}

 \section{The semigroup generated by $ \mathfrak H,  \mathfrak C$
 and geometrically special elliptic curves}
 
 An easy but important observation, in order to study the semigroup $\sW(\frak H, \frak C)$,
  is moreover that $\infty =(0,1)$ is  a fixed point for both  $\mathfrak H$ and $\mathfrak C$.
 
 The other fixpoints for $\mathfrak H$ are the roots of 
 $$ (2 \la)^3 = -1,$$
 hence the four fixpoints of $\mathfrak H$ correspond to the four triangles in the Hesse pencil.
 
 Moreover $\la = - 1/2$  and $\infty$ are     the only real fixed points  of $\frak H$.
 
On the other hand the  other fixpoints for $\mathfrak C$ are the roots of 
 \begin{equation}\label{E:fixptsC}
  4 \la^3 + 6 \la^2 - 1 =0
   \end{equation}
 an equation which is solved by  the three real roots  $-\frac{1}{2}, \ -\frac{1}{2}\pm \frac{\sqrt{3}}{2}$.

We  observe here the  remarkable coincidence that the real fixpoints of  $\mathfrak H^{\circ 2}$ are the fixpoints
of $\mathfrak C$, and  that, as we shall calculate,  the fixpoints of $\mathfrak H$ are the ramification points of $\mathfrak C$,  namely, the points in the set $\sT$ of triangles.

In the sequel we shall also use the property, already observed in \cite{Art-Dolg},
that both maps $\mathfrak H$, $\mathfrak C$ are {\bf critically finite}, that is,
for each of them  the {\bf Postcritical set},  defined as  the forward orbit of the set of ramification  points, is a finite set.

\begin{remark}\label{crit-fin}
(i) The set of ramification points of $\mathfrak C$ is the set $\sT$ of triangles,
and  $\mathfrak C $ fixes $\infty, -\frac{1}{2}$, while it exchanges the other two: hence the
Postcritical set  of $\mathfrak C$  is equal to $\sT$.

(ii) The set of ramification points of $\mathfrak H$ is instead the set $\sF$ of Fermat cubics
which is sent (since $\mathfrak H (0) = \infty$) to the set $\sT$, which is the set of fixed points. Hence the Postcritical set  of $\mathfrak H$ equals $\sF \cup \sT$.
\end{remark}

 Observe that both maps  $\mathfrak H$ and  $\mathfrak C$ 
  are expressed by rational functions with integral coefficients, hence they  send the real line $\PP^1(\RR)$ to itself ( and also $\PP^1(\QQ)$ to itself)
 and that the preimages of $\infty$ are for $\mathfrak H$ the point $\la=0$ with multiplicity two
 and $\infty$ with multiplicity one, respectively for $\mathfrak C$ the same two points with exchanged multiplicity (this observation shows immediately that $ \mathfrak H $ and $ \mathfrak C$ do not commute, since  both compositions send $0$ to $\infty$, but one with multiplicity 4, the other with multiplicity 1).
 
 Take now a local coordinate around $\infty$, $u : = 1/\la$.

Then the respective Taylor expansions in this local coordinate are: 
\begin{equation}\label{H-infty}  u \mapsto   1/ \mathfrak H (1/u)  = - 3 u( 1 + 1/2 (u^3))^{-1},\end{equation}
\begin{equation}\label{C-infty} u \mapsto     1/  \mathfrak C (1/u)   = - \frac{3}{2} u^2  ( 1 - 1/4 (u^3))^{-1}.\end{equation}

 These local calculations allow us to give another proof  that $ \mathfrak H $ and $ \mathfrak C$ do not commute, since by composition we get:

 $$ \mathfrak H  \circ \mathfrak C:  u \mapsto (- \frac{3}{2} u^2 + o (|u|^2)  ) \mapsto  \frac{9}{2} u^2 + o (|u|^2),$$
 respectively 
 $$ \mathfrak C  \circ \mathfrak H:  u \mapsto (- 3 u + o (|u|)   \mapsto  - \frac{27}{2} u^2 + o (|u|^2).$$

 The previous calculation allows us to   prove the following Proposition,
 a first step in guessing the following
 
 \begin{theorem}\label{free-semigroup}
The   homomorphism $\psi$ is injective,  that is, $ \mathfrak H $ and $ \mathfrak C$
 generate a free rank 2 semigroup.
 \end{theorem}
 
 \begin{proposition}\label{free}
 Consider a word $W(h,c) $ of exponent $e$, and with exponents $e(c)$ and $e(h)$
 of the respective letters $c,h$. 
 
 Then $f : = \psi (W(h,c) ) = W (\frak H, \frak C)$ determines the exponents 
 $e(c)$ and $e(h)$ (and their sum $e$).
 
 Moreover, if we write such a word in the normal form 
 $$ h ^{a_0} c \ h ^{a_1}   \  \dots  c \ h ^{a_j}   \dots c\ h ^{a_{e(c)}}, \ a_i \geq 0, \ \sum_i a_i = e(h),$$
 then the Taylor expansion at $\infty$ has order  $2^{e(c)}$,
 and leading coefficient equal to 
 $$ \pm (\frac{3}{2})^{ 2^{e(c) -1} } (3)^ { \sum_i 2^i a_i}.$$
 \end{proposition}
 
 \begin{proof}
 We have already observed that the degree of  $f : = \psi (W(h,c) )$ equals  $3^e$, where $e$ is the exponent (that is, the length) of the
word $W(h,c) $.  Hence the exponent $e$ is determined by $f$.

On the other hand, the order of $f$ at the fixed point $\infty$ equals $2^{e(c)}$,
where $e(c)$ is  the exponent of
$c$ in the word $W(h,c) $. 

Therefore $f$ determines the exponents $e$ and $e(c)$, hence also $e(h) = e - e(c)$.

Given  $e(h) = e - e(c) $ and $e(c)$, all the words with these exponents can be written 
in the above normal form.

\bigskip

\bigskip

 \end{proof}
 
 The previous Proposition is not yet sufficient to show the injectivity of $\psi$, as the following example shows.

 \begin{remark}
 (i) The two following words have leading coefficients of the same absolute value
 $ (\frac{3}{2})^7 (3)^{10}$:
 $$ chchhc, \ hhccch,$$
 since $ 10= 2 + 4 + 4 = 1 + 1 + 8$
 (respectively $a_0=0, a_1=1, a_2=2, a_3=0$ for the first word,
 $a_0=2, a_1=a_2=0, a_3=1$ for the second word).
 
 (ii) Looking at the Taylor development at $- \frac{1}{2} $ does not give further information.
 
 (iii)  The absolute value of the first coefficient becomes strictly larger once we exchange a product $\mathfrak H  \circ \mathfrak C$
with a product $\mathfrak C  \circ \mathfrak H$.
Therefore the absolute value  of the coefficients is strictly monotone increasing as we 
run  from 
$\mathfrak H ^{e - e(d)}  \circ \mathfrak C^{e(d)}$ until 
$ \mathfrak C^{e(d)}  \circ \mathfrak H ^{e - e(d)} $.
The problem is that we do not get all the possible words via a linear sequence of these procedures. 
 \end{remark}
 
 Therefore we change our strategy and prove a slightly stronger result. 
 
  \begin{remark}
  More generally, in view of the identities 
  $$\iota^{\circ2} = Identity, \ \frak C = \frak H \circ \iota , \  \frak H = \frak C \circ \iota$$
  we can consider the semigroup $\sW'$ of positive words in
$\ZZ * (\ZZ/2) $, which we view (see for instance \cite{derham})  as 
the set of reduced words in  $h , i$ for the relation $i^2 \equiv  \emptyset$
(here $\emptyset = $ emptyword)
or, similarly, we  consider the semigroup $\sW''$ of 
the reduced words in  $c , i$ modulo the relation $i^2= \emptyset$ .

Now,  the relation $c = h i $ yields a 
 bijection between 
the free semigroup $\sW (c,h) $
and the subset of elements of $\sW' $ which do not begin with $i$,
and similarly a bijection with  the subset of elements of $\sW''$ which do not begin with $i$.

\end{remark}

 \begin{theorem}\label{free-semigroup-2}
The homomorphisms 

$$\psi : \sW'  \ra End(\PP^1) , \ \sW'' \ra End(\PP^1),$$

defined by 
$$  \psi (i) = \iota , \psi (h) = \frak H, \psi (c) = \frak C,$$
( recall that  $  \iota (\la) =  -1/(2\la)$)
are  injective. 

In particular Theorem \ref{free-semigroup} holds true.
\end{theorem}

\begin{proof}

The advantage here is that $\iota$ exchanges $0, \infty$ 
while $\frak C (0) =  \frak H (0) = \infty,$ and $ \infty$  is a fixed point 
of $\frak C, \frak H$.

Observe preliminarily that it is sufficient to establish the result for 
$\sW'$, since  $\sW'$ and $\sW''$ are isomorphic setting
$ c=hi \Leftrightarrow h = ci$, and $\psi (\sW') = \psi ( \sW'')$.

Hence from now on we shall only consider the semigroup $\sW'$.

Before continuing with the proof, we establish an   intermediate result
which could   also be useful for other purposes.

\begin{lemma}
Given the  set 
$$ S : = \{ q \cdot 2^{\frac{b}{3}}| b \in \ZZ \setminus 3 \ZZ , \ q \in \QQ \setminus \{0\}, q = \frac{\al}{\be} , \al \in \ZZ, \be \in \NN, \al, \be \  {\rm odd \ and  \ relatively \ prime} \}$$
we have that  
$$\iota (S) = S, \ \  \frak H (S) \subset S \cup \{0\}.$$

Moreover 

(i) For $\la \in S$,  $ \frak H (\la)=0$ if and only if $\la = - 2^{\frac{-1}{3}}$ 
(that is, $ q= - 1, b=-1$).

(ii) $ \iota (- 2^{\frac{-1}{3}}) =  2^{\frac{-2}{3}}$.

(iii) For $\la \in S$ 
$$\frak H (\la) \neq - 2^{\frac{-1}{3}}, 2^{\frac{-2}{3}}.$$

(iv) For $\la \in S$ and $ f =  \psi (W'(h,i))$, where $W'(h,i) \in \sW'$,
we have $ f ( - 2^{\frac{-1}{3}}) \in \{ 0, \infty\}$
if and only if the last letter of the word $W'(h,i)$ equals $h$.

\end{lemma}
\begin{proof}
First we verify that $S$ is invariant, in fact:
$$ \iota (  q \cdot 2^{\frac{b}{3}}) = - q ^{-1} 2^{\frac{-3-b}{3}} = - \frac{\be}{\al}2^{\frac{-3-b}{3}}.$$
$$ \frak H (  q \cdot 2^{\frac{b}{3}}) = - \frac{1 + q^3 2^{b+1}}{6 q^2 2^{\frac{2b}{3}}}=
- 2^{\frac{-3 - 2b}{3}} (\frac{\be^3 + 2^{b+1}  \al ^3 }{ 3 \al^2 \be  }).$$
Observe that $ q'' := \frac{\be^3 + 2^{b+1}  \al ^3 }{ 3 \al^2 \be  }$ is not yet of the desired form, but it can be written as 
$$q''= \frac{\al'}{\be'} 2^n = : q' 2^n $$ for some $n \in \ZZ$ and 
$\al', \be' $  odd, and relatively  prime integers, as follows:

$$ b \geq 0 \Rightarrow n=0, \ b \leq -2 \Rightarrow n= b+1 \leq -1.$$

If $ b = -1$ then $ n \geq 1$.

The conclusion is that we have the normal form:
$$ \frak H (  q \cdot 2^{\frac{b}{3}}) = 
- 2^{\frac{-3 - 2b + 3n}{3}} \ \frac{\al'}{\be'} = - q' \ 2^{\frac{-3 - 2b + 3n}{3}} .$$

(i) follows now from the fact that $\frak H (\la)=0$ if and only if $ 2 \la^3 + 1=0$.

(ii) follows right away from the previous formula.

(iii-I) Assume that $\frak H (\la) =  - 2^{\frac{-1}{3}}$.

Then we have, looking at the exponent of $2$:

$$ -1 = -3 - 2b + 3n = - 2b + 3 (n-1).$$
For $ b \geq 0$ we have $n=0$, hence this case is impossible.

For $ b \leq -2$ we have $n = b+1$, hence this case is also impossible,
( $-1= b$).

For $b=-1$, we have $ n \geq 1$, hence this case is also impossible,
since  otherwise we should  have $-3 = 3 (n-1) \geq 0$.

(iii-II) Assume instead that $\frak H (\la) =   2^{\frac{-2}{3}}$.
Then we should have
$$ -2 = -3 - 2b + 3n = - 2b + 3 (n-1).$$

For $ b \geq 0$ we have $n=0$, hence  this case is impossible.

For $b=-1$, we have $ n \geq 1$, hence this case is also impossible,
since  we must have $-4 = 3 (n-1) \geq 0$.

For $ b \leq -2$ we have $n = b+1$, hence we have 
 $b = -2$, $ n = -1$ and we also have
 $$   - 2 q''  = 1 \Leftrightarrow 2 \be^3 + \al^3 = - (3 \al^2 \be).$$
 
 Since $\al, \be$ are relatively prime odd integers, the formula
 implies that $\be$ divides $\al$, and $\al$ divides $\be$, which is a contradiction
 unless $  q = \pm  1$.  Hence  in this case we would have  $
 \be=1$, $\al = \pm 1$ and the above equation
yields a contradiction.

Hence (iii) is proven.

Let us show (iv). 

By (i) $\frak H (\la)=0$ for  $ \la = - 2^{\frac{-1}{3}}$.

Then we observe that  $\iota$ and $\frak H$ leave the set $\{0, \infty\}$
invariant. 

Conversely, if the last letter of the word $W'(h,i) \in \sW'$ yielding $f$
(that is, such that $ f =  \psi (W'(h,i))$)    is not equal to $h$, it must be $i$.

We want to derive a contradiction from the property that 
$ f (  - 2^{\frac{-1}{3}}) \in \{0, \infty\}$. For this purpose we may assume that the
word $W'(h,i)$ is minimal with this property, hence 
that  $f$  cannot be written as $f' \circ f''$, such that 
$f'' ( - 2^{\frac{-1}{3}}) \in \{0, \infty\}$.

Then there exists $F \in \sW' (\frak H, \iota)$ such that we can write $ f = \frak H \circ F \circ \iota$, and we 
observe again  that $\iota ( - 2^{\frac{-1}{3}}) = 2^{\frac{-2}{3}}$.

Hence, by (i),  
we have $ F  (  2^{\frac{-2}{3}}) =  - 2^{\frac{-1}{3}}$.

There are two cases: $ F = \iota \circ \frak H \circ \Psi$,
or $ F= \frak H  \circ \Psi$ (with $ \Psi \in \sW' (\frak H, \iota)$).

In the second  case $\frak H \circ \Psi (  2^{\frac{-2}{3}}) =- 2^{\frac{-1}{3}} $,
contradicting (iii), 
in the first  case 
$\frak H \circ \Psi (  2^{\frac{-2}{3}}) = 2^{\frac{-2}{3}} $
contradicting again (iii).

Hence also (iv) is proven.
 
\end{proof}

Let's finish now the proof of the Theorem.

We know that, by Proposition \ref{free}, the image $f = \psi (W'(h,i))$ of the word $W'(h,i)$ determines the exponents of the two letters $h,i$.

Hence we can prove the result by induction on the length of the letter.

We know from (iv) of the previous Lemma that the property whether the word $w$ ends or does not end 
with $i$ is a geometric property of $f : = \psi (w)$.

If therefore $\psi (w_1) = \psi(w_2)$, then  either they both end with $i$, or they both end with $h$.

In the first  case we  can write $ w_j = w'_j i$, and it follows that $\psi(w'_1) = \psi(w'_2)$,
and we conclude the proof by induction.

In the second case    we  can write $ w_j = w'_j h$, and we conclude again 
the proof by induction since again we must  have 
 $\psi(w'_1) = \psi(w'_2)$.

In fact, the cancellation $f_1 \circ \frak H = f_2 \circ \frak H \Rightarrow f_1 = f_2$
follows because $\frak H$ is onto.

 \end{proof}

 \begin{defin}
 We shall say that an elliptic curve $E$ in the Hesse pencil is {\bf geometrically special}
 if it is a fixpoint for some endomorphism $ f \in \sW (h,d)$.
 
 Similarly we shall say that a cubic curve $E$ is geometrically special if it is isomorphic to
 such a curve in the Hesse pencil.
 
  We shall say that an elliptic curve $E$ in the Hesse pencil is {\bf geometrically pre-special}
 if its orbit under the semigroup   $  \sW (h,d)$ contains a geometrically special elliptic curve.

 We shall say that such a cubic is {\bf ergodic} if it is not geometrically pre-special.
 \end{defin}
 
 In particular, the $\mathfrak C$- and $\mathfrak H$-periodic 
 cubics are geometrically special.
 
 \begin{proposition}
 The geometrically special and pre-special cubics in the Hesse pencil form a countable set,
 which is contained in the set of algebraic points $\la$ defined over $\QQ$.

In particular, the set of ergodic points is dense in the Hausdorff topology
 (by Baire's theorem).
 \end{proposition} 
 
 \begin{proof}
 All the rational functions $f$ in the semigroup $\sW(\frak H, \frak C)$ have
 integral coefficients. Hence their fixed points are algebraic numbers,
 as well as their counterimages, 
 hence they  are contained in a countable set. 
 
 Moreover, for each element $f$ in the semigroup the degree is at least $3$,
 hence the number of its periodic points is infinite.
 
 The rest is an obvious consequence of Baire's theorem.
 
 \end{proof}

 \subsection{Another  approach for showing that $\sW(\frak H, \frak C)$ is a free semigroup}

 Umberto Zannier suggested that   the methods 
 introduced by Ritt in \cite{ritt}  concerning compositions  of polynomial maps might yield a   proof of 
 a stronger result than 
 Theorem \ref{free-semigroup}, and asked  the following question.
 
 \smallskip
 
 \textbf{Question:}
 Is it true that there are no  nonconstant rational functions $f, F$  such that 
 $$ \frak H \circ f = \frak C \circ F ?$$

 The answer is  negative, due  to the fact that $ \frak C = \frak H \circ \iota$,
 hence $F$ and $f : = \iota \circ F$ yield a solution for every choice of a nonconstant rational map $F$.

On the other hand, finding a  solution pair is equivalent to finding   a rational parametrization
 of an irreducible component of the curve $ \Ga \subset \PP^1 \times \PP^1$, defined as
 $$ \Ga : = \{ (x,y) | \ \frak H (x) = \frak C (y) \}.$$
 
We have  the following Lemma, whose proof shows how we could have found
the factorization  $ \frak C = \frak H \circ \iota$ even without having read about it in the classical texts. 
 
  \begin{lemma}
The curve  
$$ \Ga \subset \PP^1 \times \PP^1, \  \Ga : = \{ (x,y) | \ \frak H (x) = \frak C (y) \},$$
 has exactly $4$ ordinary double points as singularities, and it is reducible as the union of the graph $\Ga_{\iota}$ of the involutive projectivity $\iota (x) =  -\frac{1}{2x}$ with a smooth curve $C'$ of bidegree $(2,2)$, which has genus $1$.

The equation of $\Ga$ in affine coordinates equals 
$$  ( \iota (y) -x) (  \iota (y) + x - 2 x^2  \iota (y)^2) =0 \   \Leftrightarrow (2 xy +1) (-y + 2x y^2 - x^2)= 0.$$

 \end{lemma}
 \begin{proof}
 $\Ga$ is   a curve of bidegree $(3,3)$, hence  its arithmetic genus $p(\Ga)$
 satisfies $ 2 p(\Ga) - 2= 6$, hence $p(\Ga) = 4$.
 
We shall  show that
 $\Ga$ has $4$ singular  points, which are ordinary double points (nodes). 
 
 The singular points must 
 be pairs $(x,y)$ where $x$ is a ramification  point for $ \frak H$, 
 $y$ is a ramification point for $ \frak C$.
 
 Hence $ x \in \sF$, $y\in \sT$: and, as we observed, $ \frak H$
 yields a bijection between $\sF$ and $\sT$, while $\frak C$
 leaves $\sT$  fixed  and yields a bijection
 of $\sT$ onto itself.
 
 The conclusion is that, since for the singular points we must moreover have
  $\frak H (x) = \frak C (y)$, the curve $\Ga$ has exactly $4$ singular points; and at these
 the respective local equations of $\frak H (x) , \frak C (y)$, vanish of order exactly two.
 Hence the curve has exactly $4$ nodes as singularities.
 
 $\Ga$ is reduced, and if  it were reducible, since it does  have neither horizontal nor vertical components,  it would have a  decomposition into irreducible components of type $(1,1) + (2,2)$ or $(1,2) + (2,1)$: because if it had
 three components of bidegree $(1,1)$ it would have $6$ singular points and not $4$.
 
 Let us first consider the possibility (1,1) + (2,2).
 
 Since the curve $\Ga$  is a real curve, the component of type $(1,1)$ would be real, 
   hence it would be the graph of a real projectivity of $\PP^1(\RR)$ to itself.
 
 Observe that  the points of $\sF$ map to $\sT$ as follows under $\frak H$:
 $$ [ 0, 1, \e, \e^2 ] \mapsto [ \infty , -\frac{1}{2},  -\frac{\e}{2},-\frac{\e^2}{2}]$$
 while the points of $\sT$ map to $\sT$ as follows under $\frak H$:  
 $$ [ \infty , -\frac{1}{2},  -\frac{\e}{2},-\frac{\e^2}{2}] \mapsto [ \infty , -\frac{1}{2},  -\frac{\e^2}{2},-\frac{\e}{2}].$$
 Indeed the ordered 4-tuples 
 $ [ 0, 1, \e, \e^2 ]  $ and $ [ \infty , -\frac{1}{2},  -\frac{\e^2}{2},-\frac{\e}{2}]$
are  projectively equivalent, hence there is exactly one  projectivity sending the first
four ordered points to the other ones, namely $\iota :  x \mapsto -\frac{1}{2x}$.

Hence the component of type $(1,1)$ should be the graph of $\iota$.

 We recall now  the property that $ \frak C = \frak H \circ \iota$,
which shows that the equation $\frak H (x) = \frak C (y)$  is satisfied 
for $ y = \iota x$.

Hence $\Ga$ is the union of the  curve $\Ga_{\iota} : = \{  2 x y + 1=0\}$, and of a curve
of type $(2,2)$ whose equation is derived as follows

$$ \frak H (x) - \frak C (y) =  \frak H (x) -  \frak H (\iota (y)) = 
( \iota (y) -x) (  \iota (y) + x - 2 x^2  \iota (y)^2)  = (2 xy +1) (-y + 2x y^2 - x^2).$$

This factorization, shows our assertion and, by the way, excludes the possibility of a decomposition of type  (2,1) + (1,2).

 \end{proof}

 As a consequence, we have the following
 \begin{theorem}\label{ritt}
 If $f, F$   are nonconstant rational functions such that 
 $$ \frak H \circ f = \frak C \circ F $$
 then $ F = \iota \circ f$.
 
 In particular, we have cancellation: 
 $$ \frak H \circ F_1 =   \frak H \circ F_2 \Rightarrow F_1= F_2.$$

 \end{theorem}
 \begin{proof}
 The functions $ f(t), F(t)$ yield a rational map of $\PP^1$ 
 dominating  a component of $\Ga$. Since $C'$ has genus $1$, this component
 must be the graph $\Ga_{\iota}$, hence $F (t) = \iota ( f(t))$,
 as claimed.
 
 For the second assertion, set $f : = F_1$, $F : = \iota F_2$: then 
 $ \frak H \circ f = \frak C \circ F $ hence 
 $$ \iota \circ  F_2 = \iota \circ F_1 \Rightarrow F_1 = F_2.$$
 
 \end{proof}

 \begin{remark}
We observe now  how  a second proof of Theorem \ref{free-semigroup-2} 
 follows  from Theorem \ref{ritt}.

Assume that we have two words $w_1(h,i), w_2(h,i)$
such that $w_1(\frak H,\iota)= w_2(\frak H,\iota)$.

The exponent $e(h)$ of the letter $h$ is equal for both words  since the degree
of the corresponding rational map is $3^{e(h)}$.

We  show now  that $w_1=w_2$,  by induction on $e(h)$.
 
In the case where  $w_1 = i w'_1$, $w_2 = i w'_2$  
it suffices to multiply by $i$ on the left and reduce (because $\iota$ is invertible)
to the case where both words begin with $h$.

The second assertion of Theorem \ref{ritt} says now that if 
$w_1 = h w'_1$, $w_2 = h w'_2$, then $w'_1(\frak H,\iota)= w'_2(\frak H,\iota)$,
and we are done by induction on $e(h)$.

There remains  the case where   $w_1 = h w'_1$, $w_2 = i w'_2$.
In this case we can obviously assume that $w'_2$ is non-empty, hence $w'_2 = h w$.

Apply then the  equality of the corresponding maps to $\infty$: we have
$$ w(\frak H,\iota) (\infty), w'_1(\frak H,\iota) (\infty) \in \{0, \infty\},$$
because 
$$\frak H (\{0, \infty \} )= \{\infty \}, \ \ \iota (\{0, \infty \} )= \{0, \infty \}.$$

For the same reason, and since $\iota$ exchanges $0, \infty$,  we get
$$ w_1(\frak H,\iota) (\infty) = \infty, \ w_2(\frak H,\iota) (\infty) = 0,$$
a contradiction.

\qed

\end{remark}



\section{The geometry  of the map $\mathfrak C$ and its real periodic points}

In this section we will begin a qualitative study of the geometry of $\mathfrak C$ and of its iterates.  In this respect it is important to understand the behaviour at the fixed points of $\mathfrak C$. 

\begin{defin}[\cite{Milnor}]
Let $f:\BP^1 \ra \BP^1$ be an endomorphism of degree $m$ larger than 1. Let $p\in \BP^1$ be a fixed point of $f$ and suppose that $f$ has the following expression in a local coordinate $t$ near $p$:
$$
f(t)=a t +  bt^2 + \cdots
$$
Then $p$ is called  \emph{attracting}  (resp.  \emph{repelling}, resp. \emph{indifferent}) if $|a| < 1$ (resp. $|a| > 1$, resp. $|a|=1$). If $a=0$ then $p$ is called \emph{superattracting}.

If $p$ is an attracting fixed point, we define the \emph{basin of attraction} of $p$ as the open set $\Omega\subset \BP^1$ consisting of all points $z\in \BP^1$ such that the sequence $f(z), f^2(z), f^3(z), \dots$ of their images under successive iterates of $f$ converges to $p$.
\end{defin}

By Prop.  \ref{P:fixpts} we know that 
$$
\mathfrak C (\lambda) = \frac{1-4\lambda^3}{6\lambda}
$$
  has 4 fixed points counted with multiplicity, because it has degree 3. In fact, as we already mentioned in Section 4, they are the points
  $$
  \infty, -\frac{1}{2}, -\frac{1}{2}\pm \frac{\sqrt{3}}{2}
  $$
and therefore they have multiplicity 1.  We also computed in Section 4 that  $\mathfrak C$  has the following expression in a local coordinate $u$ around $\infty$:
$$
\mathfrak C(u)= - \frac{3}{2} u^2  ( 1 - 1/4 (u^3))^{-1}
$$
Therefore $\infty$ \emph{is a superattracting fixed point}  ($a=0$).

At the point $-\frac{1}{2}$ we can write  $\lambda +\frac{1}{2} = v$ and we get the following expression in the coordinate $v$: 
$$
 \widetilde{\mathfrak C}(v):=\mathfrak C(v-\frac{1}{2}) + \frac{1}{2}= \frac{3v^2-2v^3}{3(v-\frac{1}{2})} = \frac{v^2(1-\frac{2}{3}v)}{v-\frac{1}{2}}
$$
Therefore $-\frac{1}{2}$  \emph{is a superattracting fixed point}  ($a=0$).

At the point $-\frac{1}{2} -  \frac{\sqrt{3}}{2}$ we have (computed  by hand, and checked 
with WolframAlpha):
$$
 \mathfrak C(v-\frac{1}{2}(1+\sqrt{3}))+  \frac{1}{2} (1 + \sqrt{3})=\frac{(1 - 4 (v -\frac{1}{2} (1 + \sqrt{3})^3)}{6 (v -  \frac{1}{2} (1 +\sqrt{3}))} +  \frac{1}{2} (1 + \sqrt{3})= 
 $$
 $$
 = \sqrt{3} v + (1 - \sqrt{3}) v^2 + \cdots
$$
Therefore $-\frac{1}{2} -  \frac{\sqrt{3}}{2}$ \emph{is a repelling fixed point} ($a=\sqrt{3}$).

At the point $-\frac{1}{2} +  \frac{\sqrt{3}}{2}$ we have (computed  by hand, and checked with WolframAlpha):
$$
 \mathfrak C(v-\frac{1}{2}(1-\sqrt{3}))= - \sqrt{3} v + (1 + \sqrt{3}) v^2 + \cdots
$$
and therefore also $-\frac{1}{2} +  \frac{\sqrt{3}}{2}$ \emph{is a repelling fixed point} ($a=-\sqrt{3}$).

\medskip

Moreover, we have:

\begin{lemma}
The ramification points of $\mathfrak C$ are 4 points counted with multiplicity one, namely 
$\la = \infty$ and the three roots of $ (2\la)^3 = -1$ (that is, $\la=-\frac{1}{2}  \eta$
where $\eta$ is a cubic root of $1$).
\end{lemma}
\begin{proof}
By Hurwitz' formula we have four ramification points, these are the points where the two vectors
$$\mathfrak C, \frac{\partial \mathfrak C}{ \partial \la}$$
are proportional, i.e. the points where the two vectors 
 $$ (6 \la, 1 - 4 \la^3), ( 1,  -2 \la^2)$$
are proportional.

\end{proof}

We use the above information to show the  following.

\begin{theorem}\label{D-periodic}
The real periodic points of $\mathfrak C$ are just the four fixpoints 
$$
  \infty, -\frac{1}{2}, -\frac{1}{2}\pm \frac{\sqrt{3}}{2}
  $$

\end{theorem}

\begin{proof}
The real line is divided into 4 arcs, with increasing value of $\la$,
and with extremes the four fixpoints 
$$ (\infty, -\frac{1}{2} -  \frac{\sqrt{3}}{2}, -\frac{1}{2},-\frac{1}{2} +  \frac{\sqrt{3}}{2},\infty).$$
We observe the following
\begin{enumerate}
\item
Arc 1 goes bijectively onto itself  respecting the increasing order
\item
Arc 2 goes bijectively onto itself  respecting the increasing order
\item
the common extreme point of Arcs 1 and 2 is a  repelling     fixpoint, hence
points in the Arcs 1 and 2 which are extreme points are transformed by the iterates of 
$\mathfrak C$ to points which are closer to the external extreme points $ \infty, -\frac{1}{2}$;
in particular there are no other periodic points in the Arcs 1 and 2.
\item
Arc 3 is sent bijectively in a decreasing way to the sequence of Arcs 2-1-4
\item
Arc 4 is sent bijectively in a decreasing way to the sequence of Arcs 3-2-1

\end{enumerate}

Assume that $P$ is a periodic point different from the four fixpoints (extremes of the four Arcs). We have seen in item (3) that $P$  must lie either in Arc 3 or in Arc 4.

If a power  $\frak C^{\circ i}$ sends $P$ in a point of Arcs 1,2, then 
$P$ cannot be a fixpoint of  $\mathfrak C^{\circ n}$ for all $n \geq i$,  hence $P$ is not periodic.

If all transforms of $P$ remain in the Arcs 3 and 4,  we observe that the map is decreasing in the union of the two arcs,
and that if P is periodic, it may be fixed only by an even iterate of $\frak C$.

Moreover, if $\frak C^{\circ 2i} (P) = P$ then, setting $Q : = \frak C(P)$ , it follows that also 
$$ \frak C^{\circ 2i} (Q) = \frak C^{\circ 2i+1 } (P) = \frak C (\frak C^{\circ 2i} (P)) = \frak C(P) = Q.$$

Hence it suffices to consider periodic points   $P$  lying  in Arc 4, and  the action of even iterates on them, assuming that each $\frak C^{\circ 2}$  continues 
to  send the point again to Arc 4.

We observe now that  $\frak C^{\circ 2}$ sends a small  subarc starting from the fixed   point 
$-\frac{1}{2} +  \frac{\sqrt{3}}{2}$ 
bijectively to Arc 4, then the rest of Arc 4 
is sent to the   arcs 1 and 2.

Since there are no fixed points for $\frak C^{\circ 2}$ in the open Arc 4, as we showed in Corollary \ref{harmonic}, and the derivative of $\frak C^{\circ 2}$ is equal to $3$ at the repelling  fixpoint $-\frac{1}{2} +  \frac{\sqrt{3}}{2}$, hence 
$\frak C^{\circ 2} (t) $ is   strictly greater than $t$ on the whole subarc, hence 
by induction all our points $P$ in this subarc continue to be sent by iterates of $\frak C^{\circ 2}$ to  points $P'$  with $ P' > P$.
Hence $P$ cannot be a periodic point, a contradiction.

\end{proof}

We have seen in the proof above that the points of the Arcs 1 and 2 have 
three real counter images, while the points of the Arcs 3 and 4 have 
only one real counter image. We get a nice picture of the geometry of $\mathfrak C$
via the following Proposition showing which are the counterimages of the Arcs 3 and 4.
\begin{proposition}
The inverse image of the real line $\PP^1(\RR)$ under $\mathfrak C$
consists of $\PP^1(\RR)$ and of a closed curve $M$ consisting of two complex conjugate 
vertical arcs joining $-\frac{1}{2}$ with $\infty$, and given by the equation 
$$ 8 x r^2 = -1,  \text{where} \ \la = x + i y, \ r^2 = x^2 + y^2 = \la \bar\la.$$

The complement of $M \cup \PP^1(\RR)$ consists of four regions, and 
the right halves of the upper, respectively lower half plane map with degree 2
to the respective half planes.

\end{proposition}
\begin{proof}
The equation 
$$\mathfrak C (\la) = \overline{\mathfrak C(\la)} \Leftrightarrow (\bar\la - \la )= 4  \la \bar\la
(\la^2 - \bar\la^2) \Leftrightarrow \la \in \RR \ {\rm or} \ 8 x r^2 = -1.$$
For $\la \in \RR$, $8 x r^2 = -1$ means $ \la = -\frac{1}{2}$.

Moreover the two non real ramification points are $- \frac{\eta}{2}$
where $\eta$ is a primitive cubic root of $1$, that is,
$$ - \frac{1}{4} (1  \pm i \sqrt{3}).$$
These points have $ -\frac{1}{16} = x r^2 > -\frac{1}{8}$.

\end{proof}

\smallskip 

{\bf The basin of attraction of } $-\frac{1}{2}$.

It is convenient to change affine coordinate setting $v: =\lambda+ \frac{1}{2}$,
so that the superattractive fixed points are $v = 0, v = \infty$.

In order to study the basin of attraction of $\infty$ we consider
 the set of points with 
 $$| \widetilde{\mathfrak C}(v)| > |v|$$

(recall that $ \widetilde{\mathfrak C}$ is the expression of $ \mathfrak C$ in the coordinate $v$), while to study the basin of attraction of $\la=-1/2$ ($v=0$)
we consider 
the set of points with
 $$| \widetilde{\mathfrak C}(v)| < |v|$$

We  study therefore  the set of  points $v=\lambda+ \frac{1}{2}$ such that $| \widetilde{\mathfrak C}(v)|=|v|$, i.e. such that 
$$
|v|\left| 1-\frac{2}{3}v\right| = \left| v-\frac{1}{2} \right|
$$
which, in terms of $\lambda$, is equivalent to:
$$\left|\lambda+\frac{1}{2}\right||2(\lambda-1)|=3|\lambda|
$$
Writing $\lambda=x+iy=R(\cos\theta+i\sin\theta)$ we get:
$$
9R^2=(4R^2+1+4x)(1+R^2-2x) \ \Leftrightarrow \  9R^2=(4R^2+1)(1+R^2)+x(2-4R^2)
$$
$$
 \Leftrightarrow \ (2R^2-1)^2 = 2x(2R^2-1)
$$
This equality holds for $R = \frac{1}{\sqrt{2}}$ or for $2R\cos\theta = (2R^2-1)$.  The first condition gives the   circle $\sC$ boundary of the disc of radius $\frac{1}{\sqrt{2}}$ and centered at $0$.  

\medskip 
Let's study the  curve $\Gamma_D$ given by the second condition, which we rewrite as
$$
\cos\theta = \frac{2R^2-1}{2R}\  \Leftrightarrow \ \left(R-\frac{1}{2}\cos\theta\right)^2= \frac{1}{4}(2+\cos^2 \theta)
$$
$$
 \Leftrightarrow \ \left|R-\frac{1}{2}\cos\theta\right| = \frac{1}{2} \sqrt{2+\cos^2 \theta}
 $$
This condition corresponds to two possibilities:

\medskip

Case (I): $R-\frac{1}{2}\cos\theta = \frac{1}{2} \sqrt{2+\cos^2 \theta} \ \Leftrightarrow \ R=  \frac{1}{2} \left[\cos\theta+\sqrt{2+\cos^2 \theta}\right]$

Case (II): $-R+\frac{1}{2}\cos\theta = \frac{1}{2} \sqrt{2+\cos^2 \theta}\ \Leftrightarrow \ R= \frac{1}{2} \left[\cos\theta-\sqrt{2+\cos^2 \theta}\right]$

\medskip
Case (II) can be excluded because the term   $[-]$ is negative.  Case (I) describes a curve  symmetric w.r. to the real axis and intersecting it at the points where $\cos\theta= \pm 1$, i.e. at the two points 
$\left(\frac{\pm \sqrt{3} + 1}{2},0\right)$.

\begin{proposition}
The interior disc $\sB$  of $\sC$ and the interior domain $\Omega_D$ with boundary $\Gamma_D$ divide the complex projective line in four regions, such that 
on the complement of $\sB \cup \Omega_D$
and on the intersection $\sB \cap \Omega_D$ we have $| \widetilde{\mathfrak C}(v)| > |v|$,
while on  $\sB \setminus \Omega_D$ and on $ \Omega_D \setminus  \sB$ we have
$| \widetilde{\mathfrak C}(v)| < |v|$.

In particular the exterior to the circle $ |v| = 1 + \frac{\sqrt{3}}{2} $
is contained in the basin of attraction of $\infty$, while the interior to the circle 
$ |v| =  1 - \frac{\sqrt{3}}{2} $
is contained in the basin of attraction of $\la = -\frac{1}{2}$.

\end{proposition}
\begin{proof}
Observe that the complementary set of the set 
 $$ \{ v | | \widetilde{\mathfrak C}(v)|=|v|\} = \sC \cup \Gamma_D$$
consists of
four connected components, and the inequality $| \widetilde{\mathfrak C}(v)| > |v|$
certainly holds on the component which contains $\infty$, and also on the component which contains $0$, since
 $\mathfrak C (0) = \infty$.

\end{proof}

\begin{remark} One can easily give some sharper estimates for domains contained 
in the respective basins of attractions.
\end{remark}

\section{The geometry  of the map $\mathfrak H$ and its real periodic points}

Recall that the fixpoints of the endomorphism $\mathfrak H$ are
$\infty$ and the three points solution of 
$$ (2 \la)^3 = -1;$$
they   correspond to the four triangles in the Hesse pencil.

Since $$\mathfrak H (\la) = - 
\frac{1 + 2 \la^3}{6 \la^2}$$
the four ramification points of $\mathfrak H$  are 
$$ \la =0, \la^3 = 1 \Leftrightarrow \la =  \eta, \ \eta^3 = 1.$$
 These correspond to the equianharmonic (Fermat) elliptic curves,
 see Section 2 (and also Remark 2.2 of \cite{ca-se},
and  Cor. 2.3 ibidem, which geometrically distinguishes them
in terms of the intersection points of the nine flex tangents).

Hence there are two real ramification points, $\la=0$, which maps to $\infty$,
and $\la = 1$, which maps to $- \frac{1}{2}$; but there are only two real fixpoints,
$$\infty, - \frac{1}{2}.$$

As we saw in formula \ref{H-infty}, $\infty$ is a repelling fixpoint, while also  the point $\la= - \frac{1}{2}$
is a repelling fixpoint: 
 if we 
take the affine coordinate $ v = \la + \frac{1}{2}$, then
$$ \widetilde{\mathfrak H }(v)  := \mathfrak H (v-\frac{1}{2})+\frac{1}{2} =\frac{ -4 v^3 + 12 v^2 -9 v}{ 3  (2v -1)^2}.$$   

Instead, since 
 $$\frac{\partial \mathfrak H}{\partial \la} = \frac{1 - \la^3}{ 3 \la^3} ,$$

at the fixpoints where $ (2 \la)^3 = -1$
the derivative equals 
$$  \frac{2^3  - 2^3 \la^3}{ 3 \cdot 2^3 \la^3} = \frac{ 2^3 + 1 }{-3} = -3 .$$

Hence we have exactly four repelling fixpoints.

We have the following.

\begin{theorem}\label{H-periodic}
 $\mathfrak H$ has infinitely many real periodic points.
 Indeed, $\mathfrak H^{\circ 2m}$  has, for $ m \geq 2$,  at least $ 2(3^m -1)$ real fixpoints.
 
  Likewise, the lines $\e \RR, \e^2 \RR \subset \CC \subset \PP^1(\CC)$ contain each infinitely many  periodic points.

\end{theorem}

\begin{proof}
The real line is divided into 4 arcs, with increasing value of $\la$,
and with extremes the four points 
$$ (\infty , -\frac{1}{2}, 0, 1, \infty).$$
We observe  preliminary that the derivative of $\mathfrak H$ equals 
$$ \frac{1 - \la^3}{ 3 \la^3},$$
 hence   $\mathfrak H$ is decreasing in the first two Arcs (from $\infty$ until $0$)
and in the Arc 4, while it is increasing in the Arc 3 (from $0$ to $1$).
Also
\begin{enumerate}
\item
Arc 1 is sent bijectively  in decreasing order to the sequence of Arcs 4-3-2
\item
Arc 2 goes bijectively onto Arc 1 in decreasing order
\item
Arc 3 is sent bijectively onto Arc 1  in an increasing way 
\item
Arc 4 is send bijectively onto Arc 1 in a decreasing way.
\end{enumerate}

It is therefore clear that  we have fixpoints other than the points $- \frac{1}{2} , \infty$,
only for iterates of $\mathfrak H$
of even order.

Define $ f : = \mathfrak H^{\circ 2}$.

Then we see that $f$ sends Arc 1 to itself, and that Arc 1  can be subdived into three subarcs
such that first goes bijectively to Arc 1 in increasing order, the second 
 goes bijectively to Arc 1 in decreasing order, the third 
  goes bijectively to Arc 1 in increasing order.
  
  In the second subarc the function $f$ decreases from $- \frac{1}{2} $ to $-\infty$,
  hence there is a fixed point of $f$: because in this subarc  $f(\la) - \la$ decreases from a strictly positive value to 
  $-\infty$, hence the value $0$ is obtained.
  
  The same pattern is easily shown to hold in the complement of Arc 1, namely
  in  the union of the other three Arcs.
  
  Hence $f$ has at least $2 + 2 =4$ fixpoints.
  
  Similarly, the map $f^{\circ m}$ sends Arc 1 to itself but it zig-zags up and down $3^m$-times, hence 
  it has other $3^m -2$ fixpoints on Arc 1, and similarly for the complementary arc.
  
  In whole,  $f^{\circ m}$ has, for $ m \geq 2$, at least 
  $$ 2 + 2 (3^m -2)  = 2 ( 3^m -1)$$
  fixpoints on $\PP^1(\RR)$.
  
   The proof of the second assertion follows easily from the fact 
  that $\mathfrak H$ commutes with every element $g$ of the
  Alternating Hesse group $\mathfrak A_4$.
  
  Indeed then
  $$ g \circ \mathfrak H^{\circ n} =  \mathfrak H^{\circ n} \circ g,$$
  and the periodic points of period $n$
  consist then of a union of  $\mathfrak A_4$-orbits.
  
  We conclude since $\mathfrak A_4 \supset \{ 1, \ga, \ga^2\}$
  and $\ga (\la) = \e \la$.

  \end{proof}


\section{Julia sets of endomorphisms in the semigroup $\sW(\frak H, \frak C)$ and  directions of further research}

As we have seen,  the fixpoints of the Hessian map $\frak H$ are all repelling and  the geometry of the iterates of  $\frak H$
is somehow particular:  this  reflects  on the structure of  its  set of periodic points.

Instead, for all other elements of the semigroup $\sW(\frak H, \frak C)$ we have that
the two fixpoints  
$\infty, - \frac{1}{2}$ are superattractive, hence one can in principle compute the
respective
basins of attraction, as we did for $\mathfrak C$.

 These basins of attractions are contained in the Fatou set, 
the maximal set on which the iterates form a normal family
(see \cite{Milnor}, \cite{mcmullen}),  hence the Fatou set is  not empty.
Recall that the Julia set is the complementary set of the Fatou set (see \cite{Milnor}, \cite{mcmullen}).

Recall that a rational map $f$ is said to be critically finite if
the Postcritical set of $f$, which is defined as
$$ P(f) : = \cup_{i\geq 1} f^{\circ i} (Ram (f)), $$
(here $ Ram(f)$ is the set of ramification points of $f$, also called the critical set of $f$)
is a finite set.

For critically finite rational maps we have the following result, see Theorem A.6 , page 186 of \cite{mcmullen}:

\begin{theorem}\label{mm} 
Let $f$ be a critically finite rational map. Then each periodic cycle 
is either repelling or superattracting.

If $f$ possesses no superattracting cycles, then the Julia set $J(f)$ 
of $f$ equals the whole projective line $\PP^1(\CC)$.

\end{theorem}

We have already shown (Remark \ref{crit-fin}) that $\frak H, \frak C$ are critically finite. Hence we derive the following result.

\begin{theorem}\label{julia}
The Julia set $J(\frak H)$ 
equals the whole projective line $\PP^1(\CC)$.

Also, for each endomorphism  $ f \in \sW (\frak H, \frak C)$, 
the Julia set $J(f)$ equals the whole projective line $\PP^1(\CC)$
if and only if $f$ is not an iterate of $\frak H$.

\end{theorem}
\begin{proof}
In view of Theorem \ref{mm}, it suffices to show that
$\frak H$ does not possess any superattracting cycle.

Any superattractive cycle  contains  (compare \cite{mcmullen}, 3.1, page 36) the cyclic orbit  of a point $z$,
of cardinality $m$, 
such that the derivative of the m-th iterate of $f$ vanishes at $z$.

This means that the orbit of $z$ 
must contain a ramification point $w$, and in turn $z$ is in the
postcritical set $P(\frak H)$. But $P(\frak H) = \sF \cup \sT$,
and the set $\sF$ is the set of critical points, which are repelling, while
the points in $\sT$ are the fixed points, and are not ramification points.
Hence there are no superattractive cycles for $\frak H$.

For the second assertion,  observe that if  $f$ possesses a superattracting 
critical point, then its basin of attraction is a non empty open set contained in the Fatou set. Since the Fatou set  is 
the complementary set of the Julia set, the latter may not be the whole projective line.

 Finally, observe that the Julia set of an iterate of $\frak H$
is equal to the Julia set of $\frak H$:    one sees right away that the
assertion holds for the Fatou sets, in view of their definition 
(\cite{mcmullen}, Section 3.1, page 36); see however \cite{beardon}
Theorem  3.1.5 on page  51 for a full proof.

\end{proof}

 It would be interesting to determine the  Julia set also for
$\mathfrak C $, and also for the other endomorphisms in the semigroup 
$\sW(\frak H, \frak C)$ which are not iterates of $\frak H$.

One can  consider more generally the Julia set of a semigroup  \cite{HinkMar}; by obvious considerations (see  \cite{HinkMar}, Lemma 2.1), the Julia set for the whole semigroup $\sW(\frak H, \frak C)$ is equal to the whole complex projective line, and the   same holds for a subsemigroup  containing a (positive) power of the Hessian.

It is an interesting question to determine the Julia set for all subsemigroups of $\sW(\frak H, \frak C)$.

For instance, if  we consider the  subsemigroup consisting of all the elements which are not a power of the Hessian, we suspect that again the Julia set is the projective line.

\bigskip

Also the location of the points $\la$ of geometrically special elliptic curves
is an interesting question.
\bigskip

{\bf Acknowledgements:} the  authors would like to  thank Michael L\"onne   for an interesting conversation, Curtis McMullen for providing a useful reference, a referee for useful comments,  and especially  Umberto Zannier for a suggestion which  led to Theorem \ref{ritt}.


\end{document}